\documentclass{elsarticle}
\biboptions{sort&compress}

\usepackage{amsfonts}
\usepackage{amsmath}
\usepackage{amssymb}
\usepackage{accents}
\usepackage{geometry}
\usepackage{graphicx}
\usepackage{bmpsize}
\usepackage{pifont}
\usepackage{natbib}
\usepackage{fleqn}
\usepackage{txfonts}
\usepackage{fancyhdr}

\usepackage{subcaption}
\usepackage{float}
\usepackage{epsfig}
\usepackage{multicol}

\usepackage{adjustbox,lipsum}
\usepackage{bm}

\usepackage{mathrsfs}
\usepackage{wrapfig,epstopdf}
\usepackage{epsfig}
\usepackage{xcolor, soul}
\usepackage{url}
\usepackage{tikz}

\newcommand{\vertiii}[1]{{\left\vert\kern-0.25ex\left\vert\kern-0.25ex\left\vert #1\right\vert\kern-0.25ex\right\vert\kern-0.25ex\right\vert}}

\newtheorem{theorem}{Theorem}
\newtheorem{corollary}[theorem]{Corollary}
\newtheorem{lemma}[theorem]{Lemma}
\newdefinition{definition}{Definition}
\newproof{proof}{Proof}

\begin{document}

\title{Unconditionally energy stable and first-order accurate numerical schemes for the heat equation with uncertain temperature-dependent conductivity}

\author[nswc]{J. A. Fiordilino\corref{cor1}}
\ead{joseph.a.fiordilino1@navy.mil}
\author[nswc,csb]{M. Winger}
\ead{matthew.winger@navy.mil}

\cortext[cor1]{Corresponding author}
\address[nswc]{Naval Surface Warfare Center Corona Division, Measurement Science and Engineering Department, Corona, CA 92860.}
\address[csb]{California State University Long Beach, Department of Mathematics and Statistics, Long Beach, CA 90815.}

\begin{abstract}
In this paper, we present first-order accurate numerical methods for solution of the heat equation with uncertain temperature-dependent thermal conductivity.  Each algorithm yields a shared coefficient matrix for the ensemble set improving computational efficiency.  Both mixed and Robin-type boundary conditions are treated.  In contrast with alternative, related methodologies, stability and convergence are unconditional.  In particular, we prove unconditional, energy stability and optimal-order error estimates. A battery of numerical tests are presented to illustrate both the theory and application of these algorithms.
\end{abstract}


\maketitle
\section{Introduction}
Demand for superior predictions of scientific and engineering problems is ever increasing.  Improvement of available computational resources and both development and application of numerical methodologies work synergistically to meet the aforementioned demand.  In particular, numerical schemes are devised to improve model accuracy (e.g., via inclusion of additional physics), replicate additional properties of the continuous problem (e.g., long-time stability), incorporate uncertainty quantification via statistical techniques, etc.  The focus of this manuscript is on improving the efficiency of ensemble simulations, which facilitate uncertainty quantification, applied to heat conduction dynamics with increased model physics.

The crisis of predictability in numerical weather prediction, led to the discovery of chaos and the use of ensemble simulations to produce predictive results with uncertainty quantified.  Some key figures include, Charney \cite{Charney}, Philips \cite{Philips}, Thompson \cite{Thompson}, Lorenz \cite{Lorenz,Lorenz2,Lorenz3}; see, e.g., \cite{Kalnay,Lewis} and references therein for a historical perspective.  Ensemble calculations typically involve $J$ solves of a set of equations with slightly perturbed initial data.  Calculations are performed as either $J$ sequential, fine mesh runs or $J$ parallel, coarse mesh runs of a given code.  The ensemble average tends to perform better as a prediction than any of the individual realizations; see, e.g., Chapter 6 Section 5 of \cite{Kalnay} or \cite{Atger,Fritsch,Kalnay2}.  Evidently, increased computational resources are needed over a single realization run.  Moreover, since both increased ensemble size $J$ and mesh density $h$ yield superior results, there is an inherent competition for available computational resources.

The last six years have seen increased focus in improving efficiency of ensemble calculations \cite{Jiang, Jiang2, Jiang3, Jiang4, Mohebujjaman, Takhirov, Gunzburger, Gunzburger2, Gunzburger3, Gunzburger4, Fiordilino, Fiordilino2, Fiordilino3, Fiordilino4, Li, Chen, Ju, Luo, Luo2,Sakthivel} and references therein. The driver for much of this work is owed to a breakthrough work by Jiang and Layton \cite{Jiang}, as applied to non-isothermal fluid flow.  Therein, they recognized that a consistent modification of the convective term, utilizing the ensemble mean and fluctuation of the viscosity together with lagging of the fluctuation term, would yield a shared coefficient matrix independent of the ensemble member $j$.  The result was a reduction in both storage requirements and solution turnover time.

Recent years have seen increased focus towards problems with uncertain parameters and considerations of alternative physics.  Of particular interest here, first- and second-order ensemble algorithms for iso-thermal fluid flow with constant viscosity were developed in \cite{Gunzburger, Gunzburger2}.  Further, first-order methods were presented for the heat equation with constant thermal conductivity under mixed boundary conditions in \cite{Sakthivel} and both space and time dependent thermal conductivity under Dirichlet boundary conditions in \cite{Luo}.  Moreover, first and second-order methods were developed for spatially dependent thermal conductivities in \cite{Fiordilino4}.  Notably, stochastics were incorporated in \cite{Luo,Luo2} via the Monte Carlo method and in \cite{Li} for the convection-diffusion equation via stochastic collocation.  

In each of the above works, both stability and convergence were conditionally dependent on the ratio between the fluctuating and mean values of the relevant parameter.  In contrast, the ensemble methods presented herein are unconditionally, nonlinearly, energy stable and first-order accurate, with $\Delta t = \mathcal{O}(h)$.  Moreover, we consider the heat equation with uncertain temperature-dependent thermal conductivity due to uncertain initial conditions.  Physically, this is more realistic as most materials' thermal conductivity exhibit non-trivial temperature-dependence.  Mathematically, the resulting equation becomes nonlinear, in the diffusive term, presenting new challenges over the analogous linear problem.

\ \ \ \  Let $\Omega \subset \mathbb{R}^{d}$ be an open, bounded, Lipschitz domain.  Given initial temperature $T^0(x)=T(x,0)$, thermal conductivity $\kappa$ and heat source $f$, find $T(x,t): \Omega \times (0,t^*]\rightarrow R$ satisfying
\begin{align}\label{1.01}
T_t-\nabla \cdot(\kappa\nabla T)=f\  in \ \Omega.
\end{align}
We consider two boundary configurations: mixed and Robin. Throughout, $\kappa$ is the thermal conductivity of the solid medium which depends on the temperature profile; that is, $\kappa \equiv \kappa(T)$.  For the mixed boundary condition, the boundary $\partial \Omega$ is partitioned such that $\partial\Omega=\overline{\Gamma_D}\bigcup\overline{\Gamma_N}$ with $\Gamma_D\bigcap\Gamma_N=0$ ($\Gamma_D$ for Dirichlet condition and $\Gamma_N$ for Neumann condition).  Let $n$ denote the outward normal, then
\begin{align}\label{1.02}
T=0\ on\ \Gamma_D,\ \nabla T\cdot n=0\ on \ \Gamma_N.
\end{align}
Moreover, the Robin condition is prescribed via
\begin{align}\label{1.03}
\alpha T+\kappa \nabla T\cdot n=\beta\  on\ \partial \Omega,
\end{align}
where $\alpha \in [0,1]$ is the emissivity, and $\beta$ a prescribed function on the boundary.

The paper is organized as follows. In Section 2, we introduce mathematical preliminaries required in the analysis, including semi-discrete numerical schemes and finite element preliminaries. The fully discrete schemes are introduced in Section 3. Sections 4 and 5 are devoted to the stability and error analysis of the fully discrete algorithms. Results from a battery of numerical tests are provided in Section 6.  These serve to illustrate the validity of the proven theory and value of the algorithms.  Finally, conclusions are drawn in Section 7.
\section{Mathematical preliminaries}
Herein, we introduce notation and preliminaries that are necessary for presentation and analysis.  $H^{s}(\Omega)$ denotes the Hilbert space of $L^{2}(\Omega)$ functions with distributional derivatives of order $s\geq 0$ in  $L^{2}(\Omega)$.  The corresponding norms and seminorms are $\| \cdot \|_{s}$ and $\vert \cdot \vert_{s}$. In the special case $s = 0$, $H^{0}(\Omega) = L^{2} (\Omega)$ and the associated inner product and induced norm are $(\cdot , \cdot)$ and $\| \cdot \|$.  Moreover, $(.,.)_{\partial \Omega}$ and $\|\cdot\|_{\partial\Omega}$ denote the $L^2(\partial \Omega)$ inner product and induced norm on the boundary.  

Define the Hilbert spaces,
$$
X:=H^1(\Omega),\ Y:=\{S\in H^1(\Omega): S=0\ on \; \Gamma_D\},
$$
with dual norm $\|\cdot\|_{-1}$ understood to correspond to either $X$ or $Y$. 
The following Poincar\'{e}-like inequalities are critical in the analysis.
\begin{lemma}\label{lemma_poincare}
Let $\gamma$ be a linear form on $H^1(\Omega)$ whose restriction to constant functions is nonzero. Then, $\exists \; C_P > 0$ such that $\forall S \in X$
\begin{align}\label{lemma1_2.1}
C_P\|S\|_1\leq \|\nabla S\|+|\gamma(S)|.
\end{align}	
Moreover, if $S \in Y$ then $\exists \; C_{PF} > 0$ satisfying
	\begin{align}\label{lemma1_2.2}
		\|S\| \leq C_{PF} \|\nabla S\|.
	\end{align}
\end{lemma}
\begin{proof}
See Lemma B.63 pp. 490, for the former, and Lemma B.66 pp. 491, for the latter, in \cite{Ern}.
\end{proof}	

\noindent The Poincar\'{e}-Friedrichs inequality, inequality (\ref{lemma1_2.2}), guarantees that $\vert \cdot \vert_{1}$ is an equivalent norm to $\| \cdot \|_{1}$ in $Y$.  Recall, Young's inequality is given by
\begin{align}
    a b \leq \frac{\epsilon}{q} a^q + \frac{\epsilon^{-r/q}}{r} b^r, \ \ \ 1<q,r<\infty, \ \frac{1}{q} + \frac{1}{r}, \ a,b\geq 0.
\end{align}
The special case $q=r=2$ will be used throughout.

Let $\{T(x,t;\omega_j)\}^J_{j=1}$ denote the ensemble set of solution variables to equation (\ref{1.01}), with corresponding boundary conditions; $\omega_j$ parametrizes each ensemble member $j\in[1,J]$.  Then, the weak formulation of system $(\ref{1.01})$ and $(\ref{1.02})$ is: Find $T:[0,t^{*}]\rightarrow Y$ for a.e. $t\in (0,t^{*}]$ satisfying for $j=1,2,...,J$:
\begin{align}\label{weak1}
(T_t,S)+(\kappa\nabla T, \nabla S)=(f,S) \ \forall S\in Y.
\end{align}
Furthermore, the weak formulation of system $(\ref{1.01})$ and $(\ref{1.03})$ is : Find $T:[0,t^{*}]\rightarrow X$ for a.e. $t\in (0,t^{*}]$ satisfying for $j=1,2,...,J$:
\begin{align}\label{weak2}
(T_t,S)+(\kappa\nabla T, \nabla S)+(\alpha T,S)_{\partial\Omega}=(f,S)+(\beta, S)_{\partial\Omega} \ \forall S\in X.
\end{align}
Throughout, the thermal conductivity is assumed bounded and continuously differentiable such that:
\begin{align}
\label{xingzhiT}
|\kappa(T)-\kappa(S)|\leq C_{\kappa}|T-S| \ \forall S,T, 
\\ 0< \kappa_{min}\leq \kappa(S) \leq \kappa_{max}<\infty \ \forall S.\label{xingzhiTb}
\end{align}
\textbf{Remark:} Practically, $\kappa_{max}$ can be estimated a priori knowing that equation (\ref{1.01}) is elliptic satisfying a maximum principle \cite{Evans}.
A discrete Gronwall inequality will be critical in the subsequent stability and error analysis.  Let $N$ be a positive integer and set both $\Delta t = \frac{t^{\ast}}{N}$ and $t^{n} = n\Delta t$ for $0\leq n \leq N$.  Then, $[0,t^{\ast}] = \bigcup\limits^{N-1}_{n=0} [t^{n},t^{n+1}]$ is a partition of the time interval.
	\begin{lemma} \label{lemma_gronwall}
		(Discrete Gronwall Lemma). Let $\Delta t$, H, $a_{n}$, $b_{n}$, $c_{n}$, and $d_{n}$ be finite nonnegative numbers for n $\geq$ 0 such that for N $\geq$ 1
		\begin{align*}
			a_{N} + \Delta t \sum^{N}_{0}b_{n} &\leq \Delta t \sum^{N-1}_{0} d_{n}a_{n} + \Delta t \sum^{N}_{0} c_{n} + H,
		\end{align*}
		then for all  $\Delta t > 0$ and N $\geq$ 1
		\begin{align*}
			a_{N} + \Delta t \sum^{N}_{0}b_{n} &\leq exp\big(\Delta t \sum^{N-1}_{0} d_{n}\big)\big(\Delta t \sum^{N}_{0} c_{n} + H\big).
		\end{align*}
	\end{lemma}
	\begin{proof}
		See Lemma 5.1 on pp. 369 of \cite{Heywood}.
	\end{proof}
	
\noindent Lastly, the following norms are utilized in the error analysis: $\forall \; -1 \leq k < \infty$,
	\begin{align*}
		\vertiii{v}_{\infty,k} &:= \max_{0\leq n \leq N} \| v^{n} \|_{k}, \;
		\vertiii{v}_{p,k} := \big(\Delta t \sum^{N}_{n = 0} \| v^{n} \|^{p}_{k}\big)^{1/p}.
	\end{align*} 

We are now in a place to discuss the key idea of the numerical methods. Let $\kappa^{n} \equiv \kappa(T^{n})$ and $\kappa^{\prime n} \equiv \kappa_{max}-\kappa^{n}$.   Suppress the spatial discretization, apply an implicit-explicit time-discretization to the system $(\ref{1.01})$ with $(\ref{1.02})$:
\\\textbf{Algorithm 1 (a):}
\begin{align}
\frac{T^{n+1}-T^{n}}{\Delta t}-\kappa_{max}\triangle T^{n+1} + \nabla\cdot(\kappa^{\prime n}\nabla T^n)=f^{n+1},
\end{align}

\textbf{Remark:} For the Robin boundary condition $(\ref{1.03})$, the form of the above scheme is modified such that $(\alpha T^{n+1}, S)_{\partial \Omega}$ and $(\beta, S)_{\partial \Omega}$ appear on the left- and right-hand sides, respectively.

Applying a standard FEM discretization in space for the above system, we arrive at the following block linear system for each ensemble member $j$:
\begin{align}
    \big(\frac{1}{\Delta t}M + \kappa_{max}D\big)T^{n+1} = \big(f^{n+1} + \frac{1}{\Delta t}M + N_{\kappa}(T^n)\big)T^n,
\end{align}
	where $M$ is the mass matrix, $D$ is the diffusion matrix, and $N_{\kappa}(T^n)$ is the matrix associated with conductivity fluctuations.  The above linear system is equivalent to the following: Let A be the resulting coefficient matrix (independent of $j$).  Then, the following set of $J$ linear systems must be solved at each timestep:
	\begin{gather}\label{keymatrixequation}
		\begin{bmatrix}
			A
		\end{bmatrix}
		\begin{bmatrix}
			x_{1} \vert x_{2} \vert ... \vert x_{J}
		\end{bmatrix}
		=
		\begin{bmatrix}
			b_{1} \vert b_{2} \vert ... \vert b_{J}
		\end{bmatrix}.
	\end{gather}
	The matrix $A$ is symmetric positive definite (SPD) since both $\frac{1}{\Delta t} M$ and $\kappa_{max} D$ are SPD.  The system (\ref{keymatrixequation}) can be solved with efficient block solvers \cite{Feng,Gutknecht}.  Further, since only one coefficient matrix is required for computation per timestep, the storage requirement is thereby reduced.

\subsection{Finite Element Preliminaries}
	Let $\{\mathcal{T}_{h}\}_{0<h<1}$ be a family of quasi-uniform meshes with maximum element length $h = \max\limits_{K\in \mathcal{T}_{h}} h_{K}$.  We define the geometric interpolation of $\Omega$ as $\Omega_{h} = \bigcup_{K\in \mathcal{T}_{h}} K$.  Throughout, $\Omega$ is assumed to be a convex polytope so that $\Omega = \Omega_{h}$.  Let $X_{h} \subset X$ and $Y_{h} \subset Y$ be conforming finite element spaces defined as
	\begin{align*}
		X_{h} &:= \{S_{h} \in C^{0}(\overline{\Omega}_{h}) : \forall \; K \in \mathcal{T}_{h}, \;  S_{h}\vert_{K} \in \mathbb{P}_{l}(K) \} \cap X,
		\\ Y_{h} &:= \{S_{h} \in C^{0}(\overline{\Omega}_{h}) : \forall \; K \in \mathcal{T}_{h}, \;  S_{h}\vert_{K} \in \mathbb{P}_{l}(K) \} \cap Y.
	\end{align*}
	The spaces above satisfy the following approximation properties: $\forall 1 \leq l \leq k$,
	\begin{align}
	\inf_{S_{h} \in Z}  \Big\{ \| T - S_{h} \| + h\| \nabla (T - S_{h}) \| \Big\} &\leq Ch^{k+1} \lvert T \rvert_{k+1} \qquad T \in Z \cap H^{k+1}(\Omega), \label{bijin1}
	\end{align}
	and $Z = X$ or $Y$.
Throughout, $C$ denotes a generic positive constant independent of $h$ or $\Delta t$.
\section{Numerical Scheme}
Let $T^n_h$ be the fully discrete approximate solution at time level $t^n$, $\kappa^n_h=\kappa(T_h^n)$, and $\kappa^{\prime n}_h = \kappa_{max}-\kappa^n_h$.  Then, the fully discrete schemes are: 
\\\textbf{Algorithm 1:} 
\\ \textbf{(a)} Given $T_{h}^{n}\in Y_{h}$, find $T_{h}^{n+1}\in Y_{h}$ satisfying
\begin{align}
\label{quan1}
 (\frac{T_{h}^{n+1}-T_{h}^{n}}{\Delta t},S_h)+(\kappa_{max}\nabla T_{h}^{n+1},\nabla S_h) -(\kappa_{h}^{\prime n}\nabla T_{h}^n,\nabla S_h) = (f^{n+1},S_h)\ \forall S_{h} \in Y_{h}.
\end{align}
\textbf{(b)} Given $T_{h}^{n}\in X_{h}$, find $T_{h}^{n+1}\in X_{h}$ satisfying the fully discrete scheme as follows:

\begin{align}
\label{quan2}
(\frac{T_{h}^{n+1}-T_{h}^{n}}{\Delta t},S_h)+(\kappa_{max}\nabla T_{h}^{n+1},\nabla S_h) -(\kappa_{h}^{\prime n}\nabla T_{h}^n,\nabla S_h)+(\alpha T_{h}^{n+1}, S_h)_{\partial \Omega}
\\ = (f^{n+1},S_h)+(\beta, S_h)_{\partial \Omega}\ \ \forall S_{h} \in X_{h}. \nonumber
\end{align}




\textbf{Remark:} If the thermal conductivity is provided with explicit dependence on space and time, e.g., $\kappa\equiv\kappa(x,t)$, then 
$$(\kappa_{max}\nabla T_{h}^{n+1},\nabla S_h) - (\kappa_{h}^{\prime n}\nabla T_{h}^n,\nabla S_h) \leftarrow (\kappa^n_{max}\nabla T_{h}^{n+1},\nabla S_h) - (\kappa^{\prime n}\nabla T_{h}^n,\nabla S_h)$$
with $\kappa^n_{max} = \max_{1\leq j\leq J}\sup_{x\in\Omega}\kappa(x,t^n)$ and $\kappa^{\prime n } = \kappa^n_{max} - \kappa(x,t^n)$, in the above, yield unconditionally stable and first-order accurate methods.  The analysis is novel but analogous to that presented below.  Advantageously, $\kappa_{max}$ need not be estimated a priori.  Moreover, the consistency error is tighter with more relaxed requirements on solution regularity.
\section{Stability Analysis}
In Theorem \ref{theorem1}, stability of the temperature approximation is proven for Algorithm 1, both (\ref{quan1}) and (\ref{quan2}).  Later, first-order convergence is proven in Theorem \ref{theorem 2}.

\begin{theorem}\label{theorem1}
Consider Algorithm 1(a) and suppose $f\in L^2(0,t^{*};H^{-1}(\Omega))$ and $\beta \in H^{-1}(\partial \Omega)$, then 
\begin{eqnarray}
\begin{aligned}
\|T_h^N\|^2+\|\sqrt{\kappa_{max}}\nabla T^{N}_h\|^2+\sum^{N-1}_{n=0}\Big(\|T_h^{n+1}-T_h^n\|^2 + \Delta t\|\sqrt{\kappa^{\prime n}_h}\nabla (T^{n+1}_h-T^{n}_h)\|^2\Big)\\
+\frac{\Delta t}{2}\sum^{N-1}_{n=0}\|\sqrt{\kappa^n_{h}}\nabla T_h^{n+1}\|^2 \leq \|T_h^0\|^2 + \|\sqrt{\kappa_{max}}\nabla T^{0}_h\|^2 + \frac{2\Delta t}{\kappa_{min}}\sum^{N-1}_{n=0}\|f^{n+1}\|^2_{-1}.
\end{aligned}
\end{eqnarray}
Moreover, for Algorithm 1(b),
we have
\begin{eqnarray}
\begin{aligned}
\|T_h^N\|^2+\|\sqrt{\kappa_{max}}\nabla T^{N}_h\|^2+\sum^{N-1}_{n=0}\Big(\|T_h^{n+1}-T_h^n\|^2 + \Delta t\|\sqrt{\kappa^{\prime n}_h}\nabla (T^{n+1}_h-T^{n}_h)\|^2\Big)\\
+\frac{C_P^2\Delta t}{8}\sum^{N-1}_{n=0}\|\sqrt{\kappa^n_{h}}\nabla T_h^{n+1}\|^2_1 \leq \|T_h^0\|^2 + \|\sqrt{\kappa_{max}}\nabla T^{0}_h\|^2
\\ + \frac{4\Delta t}{C_P^2\kappa_{min}}\sum^{N-1}_{n=0} \Big(\|f^{n+1}\|^2_{-1}.
+2\|\beta\|^2_{-1,\partial \Omega}\Big).
\end{aligned}
\end{eqnarray}

\end{theorem}

\begin{proof}
Setting $S_{h}=2\Delta t T_{h}^{n+1}$ in (\ref{quan1}) and using the polarization identity on the first term, we have

\begin{eqnarray}
\begin{aligned}
\label{1.1}
\|T_h^{n+1}\|^2-\|T_h^{n}\|^2+\|T_h^{n+1}-T_h^{n}\|^2+2\Delta t \|\sqrt{\kappa_{max}}\nabla T^{n+1}_h\|^2 - 2\Delta t(\kappa^{\prime n}_h\nabla T_h^n,\nabla T^{n+1}_h) = (f^{n+1}, 2\Delta tT^{n+1}_{h}).
\end{aligned}
\end{eqnarray}
Now, 
\begin{eqnarray}
\begin{aligned}\label{key1}
2\Delta t \|\sqrt{\kappa_{max}}\nabla T^{n+1}_h\|^2 - 2\Delta t(\kappa^{\prime n}_h\nabla T_h^n,\nabla T^{n+1}_h) = 2\Delta t(\kappa_{max}\nabla T_h^{n+1},\nabla (T^{n+1}_h-T^{n}_{h})) + 2\Delta t(\kappa^{n}_{h}\nabla T_h^n,\nabla T^{n+1}_{h}).
\end{aligned}
\end{eqnarray}
Using the polarization identity twice and rearranging terms in equation (\ref{1.1}) yields,
\begin{eqnarray}
\begin{aligned}\label{key2}
\|T_h^{n+1}\|^2-\|T_h^{n}\|^2+\|T_h^{n+1}-T_h^{n}\|^2 + \|\sqrt{\kappa_{max}}\nabla T^{n+1}_h\|^2 - \|\sqrt{\kappa_{max}}\nabla T^{n}_h\|^2 + \Delta t\|\sqrt{\kappa^{\prime n}_h}\nabla (T^{n+1}_h-T^{n}_h)\|^2
\\ + \Delta t\|\sqrt{\kappa_h^n}\nabla T^{n+1}_h\|^2 + \Delta t\|\sqrt{\kappa_h^n}\nabla T^{n}_h\|^2 =  2\Delta t(f^{n+1}, T^{n+1}_{h}).
\end{aligned}
\end{eqnarray}
Application of Cauchy-Schwarz and Young's inequalities on the forcing term leads to
\begin{eqnarray}
\begin{aligned}
\label{1.3}
&2\Delta t(f^{n+1}, T^{n+1}_h)
\leq \frac{\Delta t}{\epsilon_1\kappa_{min}}\|f^{n+1}\|^2_{-1} + {\epsilon_1\Delta t}\|\sqrt{\kappa^n_{h}}\nabla T_h^{n+1}\|^2.
\end{aligned}
\end{eqnarray}
Drop $\Delta t\|\sqrt{\kappa_h^n}\nabla T^{n}_h\|^2$, use estimate (\ref{1.3}) with $\epsilon_1=1/2$, and rearrange terms.  Then,
\begin{eqnarray}
\begin{aligned}
\label{1.4}
\|T_h^{n+1}\|^2-\|T_h^n\|^2+\|T_h^{n+1}-T_h^n\|^2 + \|\sqrt{\kappa_{max}}\nabla T^{n+1}_h\|^2 - \|\sqrt{\kappa_{max}}\nabla T^{n}_h\|^2 + \Delta t\|\sqrt{\kappa^{\prime n}_h}\nabla (T^{n+1}_h-T^{n}_h)\|^2
\\ + \frac{\Delta t}{2}\|\sqrt{\kappa^{ n}_h}\nabla T_h^{n+1}\|^2 \leq \frac{2\Delta t}{\kappa_{min}}\|f^{n+1}\|^2_{-1}.
\end{aligned}
\end{eqnarray}
Summing from $n=0$ to $n=N-1$, we arrive at
\begin{eqnarray}
\begin{aligned}
\label{1.6}
\|T_h^N\|^2+\|\sqrt{\kappa_{max}}\nabla T^{N}_h\|^2+\sum^{N-1}_{n=0}\Big(\|T_h^{n+1}-T_h^n\|^2 + \Delta t\|\sqrt{\kappa^{\prime n}_h}\nabla (T^{n+1}_h-T^{n}_h)\|^2\Big) + \frac{\Delta t}{2}\sum^{N-1}_{n=0}\|\sqrt{\kappa^n_{h}}\nabla T_h^{n+1}\|^2
\\ \leq \|T_h^0\|^2 + \|\sqrt{\kappa_{max}}\nabla T^{0}_h\|^2 + \frac{2\Delta t}{\kappa_{min}}\sum^{N-1}_{n=0}\|f^{n+1}\|^2_{-1}.
\end{aligned}
\end{eqnarray}
Similarly, setting $S_{h}=2\Delta t T_{h}^{n+1}$ in equation (\ref{quan2}), we have
\begin{eqnarray}
\begin{aligned}
\label{1.7}
\|T_h^{n+1}\|^2-\|T_h^{n}\|^2+\|T_h^{n+1}-T_h^{n}\|^2 + 2\Delta t \|\sqrt{\kappa_{max}}\nabla T^{n+1}_h\|^2
+2\Delta t\|\sqrt{\alpha} T_{h}^{n+1}\|^2_{\partial \Omega}
\\ - 2\Delta t(\kappa^{\prime n}_h\nabla T_h^n,\nabla T^{n+1}_h) = (f^{n+1}, 2\Delta tT^{n+1}_{h})+2\Delta t(\beta, T^{n+1}_{h})_{\partial \Omega}.
\end{aligned}
\end{eqnarray}
Following the analysis above and rearranging leads to
\begin{eqnarray}
\begin{aligned}
\label{1.8}
\|T_h^{n+1}\|^2-\|T_h^{n}\|^2+\|T_h^{n+1}-T_h^{n}\|^2 + \|\sqrt{\kappa_{max}}\nabla T^{n+1}_h\|^2 - \|\sqrt{\kappa_{max}}\nabla T^{n}_h\|^2 + \Delta t\|\sqrt{\kappa^{\prime n}_h}\nabla (T^{n+1}_h-T^{n}_h)\|^2
\\ + \Delta t\|\sqrt{\kappa_h^n}\nabla T^{n+1}_h\|^2 + \Delta t\|\sqrt{\kappa_h^n}\nabla T^{n}_h\|^2 + 2\Delta t\|\sqrt{\alpha} T_{h}^{n+1}\|^2_{\partial \Omega} =  2\Delta t(f^{n+1}, T^{n+1}_{h}) + 2\Delta t(\beta, T^{n+1}_{h})_{\partial \Omega}.
\end{aligned}
\end{eqnarray}
Apply Cauchy-Schwarz and Young's inequalities on the two terms on the right hand side of (\ref{1.8}).  Then,
\begin{eqnarray}
\begin{aligned}
\label{1.9}
&2\Delta t(f^{n+1},T^{n+1}_h)
\leq \frac{\Delta t}{\epsilon_2\kappa_{min}}\|f^{n+1}\|^2_{-1} + \epsilon_2\Delta t\|\sqrt{\kappa^n_{h}} T_h^{n+1}\|^2_1,
\end{aligned}
\end{eqnarray}
\begin{eqnarray}
\begin{aligned}
\label{1.10}
	2\Delta t(\beta, T^{n+1}_{h})_{\partial \Omega} 
\leq \frac{\Delta t}{\epsilon_3\kappa_{min}}\|\beta\|^2_{-1,\partial \Omega} + \epsilon_3\Delta t\|\sqrt{\kappa_h^n} T^{n+1}_{h}\|^2_1.
\end{aligned}
\end{eqnarray}
From Lemma 1, we have $\|\sqrt{\kappa^n_h}\nabla T^{n+1}_h\|^2 + \|\sqrt{\alpha}T^{n+1}_h\|^2 \geq \frac{C_P^2}{2}\|\sqrt{\kappa^n_h}T_h^{n+1}\|^2_1$ and thus 
\begin{eqnarray}
\begin{aligned}
\frac{C_P^2\Delta t}{2}\|\sqrt{\kappa^n_h}T_h^{n+1}\|^2_1 \leq \Delta t\|\sqrt{\kappa^n_h}\nabla T^{n+1}_h\|^2 + \Delta t\|\sqrt{\alpha}T^{n+1}_h\|^2 
\\ \leq \Delta t\|\sqrt{\kappa^n_h}\nabla T^{n+1}_h\|^2 + \Big(\Delta t\|\sqrt{\kappa^n_h}\nabla T^{n}_h\|^2  + 2\Delta t\|\sqrt{\alpha}T^{n+1}_h\|^2\Big).
\end{aligned}
\end{eqnarray}
Combine estimates (\ref{1.9})-(\ref{1.10}) with $\epsilon_2 = 2\epsilon_3 =\frac{C_P^2}{4}$, the above estimate in equation (\ref{1.8}), and rearrange.  Then,
\begin{eqnarray}
\begin{aligned}
\label{1.11}
\|T_h^{n+1}\|^2-\|T_h^{n}\|^2+\|T_h^{n+1}-T_h^{n}\|^2 + \|\sqrt{\kappa_{max}}\nabla T^{n+1}_h\|^2 - \|\sqrt{\kappa_{max}}\nabla T^{n}_h\|^2
\\ + \Delta t\|\sqrt{\kappa^{\prime n}_h}\nabla (T^{n+1}_h-T^{n}_h)\|^2 +\frac{C_P^2\Delta t }{8}\|\sqrt{\kappa_h^n}T_h^{n+1}\|^2_1 \leq \frac{4\Delta t}{C_P^2\kappa_{min}}\|f^{n+1}\|^2_{-1}
+\frac{8\Delta t}{C_P^2\kappa_{min}}\|\beta\|^2_{-1,\partial \Omega}.
\end{aligned}
\end{eqnarray}
Summing from $n=0$ to $n=N-1$ leads to
\begin{eqnarray}
\begin{aligned}
\label{1.12}
\|T_h^N\|^2+\|\sqrt{\kappa_{max}}\nabla T^{N}_h\|^2+\sum^{N-1}_{n=0}\Big(\|T_h^{n+1}-T_h^n\|^2 + \Delta t\|\sqrt{\kappa^{\prime n}_h}\nabla (T^{n+1}_h-T^{n}_h)\|^2\Big)\\
+\frac{C_P^2\Delta t}{8}\sum^{N-1}_{n=0}\|\sqrt{\kappa^n_{h}}\nabla T_h^{n+1}\|^2_1 \leq \|T_h^0\|^2 + \|\sqrt{\kappa_{max}}\nabla T^{0}_h\|^2 + \frac{4\Delta t}{C_P^2\kappa_{min}}\sum^{N-1}_{n=0} \Big(\|f^{n+1}\|^2_{-1}
+2\|\beta\|^2_{-1,\partial \Omega}\Big).
\end{aligned}
\end{eqnarray}

\end{proof}

The stability result above shows that we have control over the temperature approximation in both $L^{\infty}(0,t^{\ast};L^2(\Omega))$ and $L^{2}(0,t^{\ast};H^1(\Omega))$, unconditionally.  Moreover, we see that the numerical dissipation is enhanced, compared to standard Backward Euler, with the additional term $\Delta t\sum^{N-1}_{n=0}\|\sqrt{\kappa^{\prime n}_h}\nabla (T^{n+1}_h-T^{n}_h)\|^2$.  

\section{Convergence analysis}
In this section, we first analyze the consistency errors for each numerical scheme.  Convergence is then proven at the anticipated, optimal rates.  Recall, the true solution satisfies
\begin{eqnarray}
\begin{aligned}
\label{zhen1}
(\frac{T^{n+1}-T^{n}}{\Delta t} ,S)+(\kappa^{n+1}\nabla T^{n+1},\nabla S) = (f^{n+1},S)+(\frac{T^{n+1}-T^{n}}{\Delta t}-T_t^{n+1},S)\ \ \forall S \in Y,
\end{aligned}
\end{eqnarray}
under mixed boundary conditions.  For the Robin boundary condition, the true solution satisfies
\begin{eqnarray}
\begin{aligned}
\label{zhen2}
(\frac{T^{n+1}-T^{n}}{\Delta t} ,S)+(\kappa^{n+1}\nabla T^{n+1},\nabla S)
+(\alpha T^{n+1},S)_{\partial\Omega} = (f^{n+1},S) + (\frac{T^{n+1}-T^{n}}{\Delta t}-T_t^{n+1},S)
\\ + (\beta, S)_{\partial \Omega}\ \ \forall S \in X.
\end{aligned}
\end{eqnarray}
Denote $e^n=(T^n-I_hT^n)-(T_h^n-I_hT^n)=\phi^n-\psi_h^n,$ where $I_hT^{n}$ is an arbitrary interpolate of $T^n$ and $e^n$ is the error at the time $t=t_n$; e.g., the Lagrange interpolate \cite{Ern} is common and applicable.  Letting $S = S_h \in X_h$ or $Y_h$ and subtracting (\ref{zhen1}) and (\ref{zhen2}) from (\ref{quan1}) and (\ref{quan2}), respectively, yields the error equations:
\begin{eqnarray}
\begin{aligned}
\label{ef1}
(\frac{e^{n+1}-e^n}{\Delta t},S_h)+(\kappa_{max}\nabla e^{n+1},\nabla S_h)
-(\kappa^{\prime n}_h\nabla e^n,\nabla S_h) =\xi_1(T^{n+1},S_h) \ \ \forall S_h \in Y_h,
\end{aligned}
\end{eqnarray}
\begin{eqnarray}
\begin{aligned}
\label{ef2}
(\frac{e^{n+1}-e^n}{\Delta t},S_h)+(\kappa_{max}\nabla e^{n+1},\nabla S_h)
-(\kappa^{\prime n}_h\nabla e^n,\nabla S_h)
+(\alpha e^{n+1},S_h) =\xi_1(T^{n+1},S_h) \ \ \forall S_h \in X_h,
\end{aligned}
\end{eqnarray}
where $\xi_1(T^{n+1},S_h)$ is defined as
\begin{eqnarray}\label{ly5} 
\xi_1(T^{n+1},S_h):\ =(\frac{T^{n+1}-T^n}{\Delta t}-T_t^{n+1},S_h)+((\kappa_{max}-\kappa^{n+1})\nabla T^{n+1},\nabla S_{h}) - (\kappa^{\prime n}_h\nabla T^n,\nabla S_{h}).
\end{eqnarray}
The following regularity assumptions are needed:
\begin{align}\label{jiashe1}
T\in L^{\infty}(0,t^{*};Y\cap H^{k+1}(\Omega)), \nabla T \in L^{\infty}(0,t^{*};L^{\infty}(\Omega)),
\\ T_{t} \in L^2(0,t^{*};X), T_{tt} \in L^2(0,t^{*};L^{2}(\Omega)), \nonumber
\end{align}
\begin{align}\label{jiashe2}
T\in L^{\infty}(0,t^{*};X\cap H^{k+1}(\Omega)), \nabla T \in L^{\infty}(0,t^{*};L^{\infty}(\Omega)),
\\ T_{t} \in L^2(0,t^{*};X), T_{tt} \in L^2(0,t^{*};H^{-1}(\Omega)). \nonumber
\end{align}
\begin{lemma}\label{lemma2}
	For T satisfying the system (\ref{1.01}) and (\ref{1.02}) and regularity assumptions (\ref{jiashe1}), the consistency error satisfies
	\begin{eqnarray}
	\begin{aligned}
	\label{lemma2.1}
	\|\xi_1(T^{n+1},S_h)\|^{2} \leq
   	\frac{C_P^2\Delta t^2}{2\kappa_{min}\epsilon_4}\|T_{tt}\|^2_{L^2(t^{n},t^{n+1};L^2(\Omega))} + \frac{C^2_{\kappa}}{2\kappa_{min} \epsilon_6}\|\nabla T^{n}\|^2_{\infty}\|e^n\|^2
	\\ + \frac{C^2_{\kappa}\Delta t}{2\kappa_{min} \epsilon_5}\|\nabla T^{n+1}\|^2_{\infty}\|T_t\|^2_{L^2(t^{n},t^{n+1};L^2(\Omega))} + (\epsilon_7^{-1} + \epsilon_8^{-1})\Delta t\frac{\kappa_{max}^2}{2\kappa_{min}}\|\nabla T_t\|^2_{L^2(t^{n},t^{n+1};L^2(\Omega))}
	\\ + \big(\sum^{8}_{i=4}\epsilon_i\big)\|\sqrt{\kappa^n_h} \nabla S_{h}\|^2.
   \end{aligned}
   \end{eqnarray}
   Moreover, suppose T satisfies the system (\ref{1.01}) and (\ref{1.03}) and regularity assumption (\ref{jiashe2}).  Then,
   \begin{eqnarray}
   \begin{aligned}
   \label{lemma2.2}
   \|\xi_1(T^{n+1},S_h)\|^{2}\leq
   	\frac{\Delta t^2}{2\kappa_{min}\epsilon_9}\|T_{tt}\|^2_{L^2(t^{n},t^{n+1};H^{-1}(\Omega))} + \frac{C^2_{\kappa}}{2\kappa_{min} \epsilon_6}\|\nabla T^{n}\|^2_{\infty}\|e^n\|^2
	\\ + \frac{C^2_{\kappa}\Delta t}{2\kappa_{min} \epsilon_5}\|\nabla T^{n+1}\|^2_{\infty}\|T_t\|^2_{L^2(t^{n},t^{n+1};L^2(\Omega))} + (\epsilon_7^{-1} + \epsilon_8^{-1})\Delta t\frac{\kappa_{max}^2}{2\kappa_{min}}\|\nabla T_t\|^2_{L^2(t^{n},t^{n+1};L^2(\Omega))}
	\\ + \big(\sum^{9}_{i=5}\epsilon_i\big)\|\sqrt{\kappa^n_h} \nabla S_{h}\|^2_1.
   \end{aligned}
   \end{eqnarray}
\end{lemma}
\begin{proof}
Recall, $\xi_1(T^{n+1},S_h)$ is defined by
\begin{eqnarray}
\label{lemma1.1}
\xi_1(T^{n+1},S_h):\ =(\frac{T^{n+1}-T^n}{\Delta t}-T_t^{n+1},S_h)+((\kappa_{max}-\kappa^{n+1})\nabla T^{n+1},\nabla S_{h}) - (\kappa^{\prime n}_h\nabla T^n,\nabla S_{h}).
\end{eqnarray}
Applying Taylor's Theorem with integral remainder, Lemma \ref{lemma_poincare}, and both Cauchy-Schwarz and Young's inequalities, we have
\begin{eqnarray}
\begin{aligned}
\label{lemma1.2}
&(\frac{T^{n+1}-T^n}{\Delta t}-T_t^{n+1},S_h)\leq \frac{C_{PF}^2\Delta t}{2\kappa_{min}\epsilon_4}\|T_{tt}\|^2_{L^2(t^{n},t^{n+1};L^{2}(\Omega))}+\frac{\epsilon_4}{2}\|\sqrt{\kappa^n_h}\nabla S_h\|^2.
\end{aligned}
\end{eqnarray}	
The last two terms in (\ref{lemma1.1}) can be reorganized as,
\begin{eqnarray}
\begin{aligned}
(\kappa^{\prime n+1}\nabla T^{n+1},\nabla S_{h}) - (\kappa^{\prime n}_h\nabla T^n,\nabla S_{h}) = (\kappa_{max}\nabla (T^{n+1}-T^{n}),\nabla S_{h}) + (\kappa^{n}_h\nabla T^{n},\nabla S_{h})
\\ - (\kappa^{n+1}\nabla T^{n+1},\nabla S_{h}).
\end{aligned}
\end{eqnarray}	
Adding and subtracting $(\kappa^n \nabla (T^{n+1}-T^{n}),\nabla S_{h})$ to the right-hand side of the above and rearranging yields,
\begin{eqnarray}
\begin{aligned}
\label{lemma1.3}
(\kappa_{max}\nabla (T^{n+1}-T^{n}),\nabla S_{h}) + (\kappa^{n}_h\nabla T^{n},\nabla S_{h}) - (\kappa^{n+1}\nabla T^{n+1},\nabla S_{h}) = (\kappa_{max}\nabla (T^{n+1}-T^{n}),\nabla S_{h}) 
\\ - ((\kappa^{n+1}-\kappa^{n})\nabla T^{n+1},\nabla S_{h}) - (\kappa^{n}\nabla (T^{n+1}-T^{n}),\nabla S_{h}) - ((\kappa^n-\kappa^{n}_h)\nabla T^{n},\nabla S_{h}).
\end{aligned}
\end{eqnarray}	

Using properties (\ref{xingzhiT})-(\ref{xingzhiTb}), Taylor's Theorem with integral remainder in the first term, and both Cauchy-Schwarz and Young's inequalities, leads to
\begin{eqnarray}
\begin{aligned}\label{lemma1.4}
- ((\kappa^{n+1}-\kappa^{n})\nabla T^{n+1},\nabla S_{h}) \leq C_{\kappa} \|T^{n+1}-T^n\| \|\nabla T^{n+1}\|_{\infty} \|\nabla S_{h}\| \leq \frac{C^2_{\kappa}\Delta t}{2\kappa_{min}\epsilon_5}\|\nabla T^{n+1}\|^2_{\infty}\|T_t\|^2_{L^2(t^{n},t^{n+1};L^2(\Omega))}
\\+ \frac{\epsilon_5}{2}\|\sqrt{\kappa^n_h} \nabla S_{h}\|^2,
\end{aligned}
\end{eqnarray}	

\begin{eqnarray}
\begin{aligned}\label{lemma1.5}
- ((\kappa^n-\kappa^{n}_h)\nabla T^{n},\nabla S_{h}) \leq \frac{C^2_{\kappa}}{2\kappa_{min}\epsilon_6}\|\nabla T^{n}\|^2_{\infty}\|e^n\|^2 + \frac{\epsilon_6}{2}\|\sqrt{\kappa^n_h} \nabla S_{h}\|^2.
\end{aligned}
\end{eqnarray}	

The other two estimates follow from Taylor's Theorem with integral remainder, Cauchy-Schwarz, and Young's inequality,
\begin{eqnarray}
\begin{aligned}\label{lemma1.6}
(\kappa_{max}\nabla (T^{n+1}-T^{n}),\nabla S_{h})
\leq \frac{\kappa_{max}^2\Delta t}{2\kappa_{min}\epsilon_7}\|\nabla T_t\|^2_{L^2(t^{n},t^{n+1};L^2(\Omega))} + \frac{\epsilon_7}{2}\|\sqrt{\kappa^n_h} \nabla S_{h}\|^2,
\end{aligned}
\end{eqnarray}	

\begin{eqnarray}
\begin{aligned}\label{lemma1.7}
- (\kappa^{n}\nabla (T^{n+1}-T^{n}),\nabla S_{h})\leq \frac{\kappa_{max}^2\Delta t}{2\kappa_{min}\epsilon_8}\|\nabla T_t\|^2_{L^2(t^{n},t^{n+1};L^2(\Omega))} + \frac{\epsilon_8}{2}\|\sqrt{\kappa^n_h} \nabla S_{h}\|^2.
\end{aligned}
\end{eqnarray}

Combining the above estimates (\ref{lemma1.1})-(\ref{lemma1.7}) yields the result (\ref{lemma2.1}).
For Robin boundary conditions, the first term of $\xi_1(T^{n+1}, S_h)$ is estimated as
\begin{eqnarray}
\begin{aligned}
\label{lemma1.8}
&(\frac{T^{n+1}-T^n}{\Delta t}-T_t^{n+1},S_h)\leq \frac{\Delta t^2}{2\kappa_{min}\epsilon_9 }\|T_{tt}\|^2_{L^2(t^{n},t^{n+1};H^{-1}(\Omega))}+\frac{\epsilon_9 }{2}\|\sqrt{\kappa^n_h} S_h\|_1^2.
\end{aligned}
\end{eqnarray}

Combining the above estimates (\ref{lemma1.3})-(\ref{lemma1.8}) and using $\|\nabla S_h\|\leq \|S_h\|_1$ yields the result (\ref{lemma2.2}).
\end{proof}
With the consistency error now analyzed, we can now prove the major convergence result.
\begin{theorem}\label{theorem 2}
Suppose T satisfies the assumptions of Lemma \ref{lemma2}.  Moreover, suppose $T^{0}_{h} \in Y_h$ is an approximation of $T^0$ to within the accuracy of the interpolant.  Then, $\exists \ C_{\dagger}$ such that scheme (\ref{quan1}) satisfies

\begin{eqnarray}
\begin{aligned}\label{dingli2.0}
\|e^{N}\|^2 + \Delta t\|\sqrt{\kappa_{max}}\nabla e^{N}\|^2 + \sum^{N-1}_{n=0} \big(\|e^{n+1}-e^{n}\|^2 + \Delta t\|\sqrt{\kappa^{\prime n}_h}\nabla (e^{n+1}-e^{n})\|^2\big)
\\ + \frac{\Delta t}{2}\sum^{N-1}_{n=0} \|\sqrt{\kappa_h^n}\nabla e^{n+1}\|^2 \leq C\exp(C_{\dagger})\Big\{\big(\kappa_{min}^{-1} + \kappa_{max}^2\kappa_{min}^{-1} + 1 + \Delta t\big)h^{2k+2}
\\ + \big(\kappa_{min}^{-1} + \kappa_{max}(1+\Delta t + \Delta t^2)\big)h^{2k} + \big(1 + \kappa_{min}^{-1} + \kappa_{max}^2 \kappa_{min}^{-1} + \kappa_{max}\big)\Delta t^2 \Big\}.
\end{aligned}
\end{eqnarray}

Moreover, scheme (\ref{quan2}) satisfies
\begin{eqnarray}
\begin{aligned}
\label{dingli2.01}
\|e^N\|^2+\|\sqrt{\kappa_{max}}\nabla e^{N}\|^2+\sum^{N-1}_{n=0}\Big(\|e^{n+1}-e^n\|^2 + \Delta t\|\sqrt{\kappa^{\prime n}_h}\nabla (e^{n+1}-e^{n})\|^2\Big)\\
+\frac{C_P^2\Delta t}{8}\sum^{N-1}_{n=0}\|\sqrt{\kappa^n_{h}}\nabla e^{n+1}\|^2_1 \leq C\exp(C_{\dagger})\Big\{\big(\kappa_{min}^{-1} + \kappa_{max}^2\kappa_{min}^{-1} + 1 + \Delta t\big)h^{2k+2}
\\ + \big(\kappa_{min}^{-1} + \kappa_{max}(1+\Delta t + \Delta t^2)\big)h^{2k} + \big(1 + \kappa_{min}^{-1} + \kappa_{max}^2 \kappa_{min}^{-1} + \kappa_{max}\big)\Delta t^2 \Big\}.
\end{aligned}
\end{eqnarray}	
\end{theorem}

\begin{proof}	
Using $e^{n}=\phi^n-\psi^n_h$, rearrange the error equation (\ref{ef1}) with $S_{h}=2\Delta t \psi_h^{n+1}$
\begin{eqnarray}
\begin{aligned}
\label{dingli2.4}
\|\psi^{n+1}_h\|^2-\|\psi_h^n\|^2+\|\psi^{n+1}_h-\psi_h^n\|^2+2\Delta t \|\sqrt{\kappa_{max}}\nabla \psi^{n+1}_h\|^2 - 2\Delta t(\kappa_h^{\prime n}\nabla \psi_h^n,\nabla \psi^{n+1}_h)\\
=
2\Delta t (\frac{\phi^{n+1}-\phi^n}{\Delta t},\psi_h^{n+1}) + 2\Delta t (\kappa_{max}\nabla \phi^{n+1},\psi^{n+1}_h) - 2\Delta t(\kappa_h^{\prime n}\nabla \phi_h^n,\nabla \psi^{n+1}_h)
\\- \xi_1(T^{n+1},2\Delta t \psi_h^{n+1}).
\end{aligned}
\end{eqnarray}
Recall equations (\ref{key1}) and (\ref{key2}) from Theorem \ref{theorem1}, we proceed in similar fashion so that after applications of the polarization identity we arrive at
\begin{eqnarray}
\begin{aligned}\label{dingli2.3}
\|\psi_h^{n+1}\|^2-\|\psi_h^{n}\|^2+\|\psi_h^{n+1}-\psi_h^{n}\|^2 + \|\sqrt{\kappa_{max}}\nabla \psi^{n+1}_h\|^2 - \|\sqrt{\kappa_{max}}\nabla \psi^{n}_h\|^2 + \Delta t\|\sqrt{\kappa^{\prime n}_h}\nabla (\psi^{n+1}_h-\psi^{n}_h)\|^2
\\+ \Delta t\|\sqrt{\kappa_h^n}\nabla \psi^{n+1}_h\|^2 + \Delta t\|\sqrt{\kappa_h^n}\nabla \psi^{n}_h\|^2 = 2\Delta t (\frac{\phi^{n+1}-\phi^n}{\Delta t},\psi_h^{n+1}) + 2\Delta t (\kappa_{max}\nabla \phi^{n+1},\psi^{n+1}_h)
\\- 2\Delta t(\kappa_h^{\prime n}\nabla \phi_h^n,\nabla \psi^{n+1}_h) - 2\Delta t \xi_1(T^{n+1},\psi_h^{n+1}).
\end{aligned}
\end{eqnarray}
Now, application of Taylor's Theorem with integral remainder, the Cauchy-Schwarz inequality, Lemma \ref{lemma_poincare}, and Young's inequality on the first term on the right-hand-side of (\ref{dingli2.4}) yields
\begin{eqnarray}
\begin{aligned}
\label{dingli2.5}
2\Delta t (\frac{\phi^{n+1}-\phi^n}{\Delta t},\psi_h^{n+1})
\leq \frac{C_{PF}^2}{\epsilon_{10}\kappa_{min}}\|\phi_t\|^2_{L^2(t^n,t^{n+1},L^2(\Omega))} + \epsilon_{10} \Delta t\|\sqrt{\kappa_h^n}\nabla\psi_h^{n+1}\|^2.
\end{aligned}
\end{eqnarray}
Applying the Cauchy-Schwarz and Young inequalities to the second and third terms yield
\begin{eqnarray}
\begin{aligned}
\label{dingli2.6}
2\Delta t (\kappa_{max}\nabla \phi^{n+1}, \nabla \psi_h^{n+1}) \leq \frac{\kappa_{max}^2\Delta t}{\epsilon_{11}\kappa_{min}}\|\nabla\phi^{n+1}\|^2 + \epsilon_{11} \Delta t\|\sqrt{\kappa^n_h}\nabla\psi_h^{n+1}\|^2,
\\ 2\Delta t (\kappa^{\prime n}\nabla \phi^{n}, \nabla \psi_h^{n+1}) \leq \frac{\kappa_{max}^2\Delta t}{\epsilon_{12}\kappa_{min}}\|\nabla\phi^{n}\|^2 + \epsilon_{12} \Delta t\|\sqrt{\kappa^n_h}\nabla\psi_h^{n+1}\|^2.
\end{aligned}
\end{eqnarray}

Using the above and Lemma \ref{lemma2} in equation (\ref{dingli2.3}), with $\epsilon_i = 1/20$ for $i=4,5,6,7,8$ and $\epsilon_{i} = 1/12$ for $i=10,11,12$, yields
\begin{eqnarray}
\begin{aligned}
\label{dingli2.7}
\|\psi_h^{n+1}\|^2-\|\psi_h^{n}\|^2+\|\psi_h^{n+1}-\psi_h^{n}\|^2 + \|\sqrt{\kappa_{max}}\nabla \psi^{n+1}_h\|^2 - \|\sqrt{\kappa_{max}}\nabla \psi^{n}_h\|^2 + \Delta t\|\sqrt{\kappa^{\prime n}_h}\nabla (\psi^{n+1}_h-\psi^{n}_h)\|^2 
\\ + \frac{\Delta t}{2}\|\sqrt{\kappa_h^n}\nabla \psi^{n+1}_h\|^2 + \Delta t\|\sqrt{\kappa_h^n}\nabla \psi^{n}_h\|^2 \leq \frac{12C_{PF}^2}{\kappa_{min}}\|\phi_t\|^2_{L^2(t^n,t^{n+1},L^2(\Omega))} + \frac{12\kappa_{max}^2\Delta t}{\kappa_{min}}\big(\|\nabla\phi^{n+1}\|^2 + \|\nabla\phi^{n}\|^2\big)
\\ + \frac{10C_{PF}^2\Delta t^2}{\kappa_{min}}\|T_{tt}\|^2_{L^2(t^{n},t^{n+1};L^{2}(\Omega))} + \frac{10C^2_{\kappa}\Delta t}{\kappa_{min}}\|\nabla T^{n}\|^2_{\infty}\|e^n\|^2 + \frac{10C^2_{\kappa} \Delta t^2}{\kappa_{min}}\|\nabla T^{n+1}\|^2_{\infty}\|T_t\|^2_{L^2(t^{n},t^{n+1};L^2(\Omega))}
\\+ \frac{5\kappa_{max}^2\Delta t^2}{\kappa_{min}}\|\nabla T_t\|^2_{L^2(t^{n},t^{n+1};L^2(\Omega))}.
\end{aligned}
\end{eqnarray}

Now, note that $\|e^n\|^2 \leq 2\|\psi^n\|^2 + 2\|\psi_h^n\|^2$, drop $\Delta t\|\sqrt{\kappa_h^n}\nabla \psi^{n}_h\|^2$, sum from $n=0$ to $n=N-1$, and rearrange terms to arrive at

\begin{eqnarray}
\begin{aligned}
\label{dingli2.8}
\|\psi_h^{N}\|^2 + \Delta t\|\sqrt{\kappa_{max}}\nabla\psi_h^{N}\|^2 + \sum^{N-1}_{n=0} \big(\|\psi_h^{n+1}-\psi_h^{n}\|^2 + \Delta t\|\sqrt{\kappa^{\prime n}_h}\nabla (\psi^{n+1}_h-\psi^{n}_h)\|^2\big)
\\ + \frac{\Delta t}{2}\sum^{N-1}_{n=0} \|\sqrt{\kappa_h^n}\nabla \psi^{n+1}_h\|^2
\leq \frac{20C^2_{\kappa}}{\kappa_{min}}\sum^{N-1}_{n=0}\|\nabla T^{n}\|^2_{\infty}\|\psi^n_h\|^2 + \frac{12C_{PF}^2}{\kappa_{min}}\|\phi_t\|^2_{L^2(0,t^{\ast},L^2(\Omega))} + \frac{24\kappa_{max}^2}{\kappa_{min}} \vertiii{\nabla\phi}^2_{2,0}
\\ + \frac{10C_{PF}^2\Delta t^2}{\kappa_{min}}\|T_{tt}\|^2_{L^2(0,t^{\ast};L^{2}(\Omega))} + \frac{10C^2_{\kappa}}{\kappa_{min}}\max_{0\leq n \leq N}\|\nabla T^{n}\|^2_{\infty}\big(2\vertiii{\phi}^2_{2,0} + \Delta t^{2}\|T_t\|^2_{L^2(0,t^{\ast};L^2(\Omega))}\big)
\\+ \frac{5\kappa_{max}^2\Delta t^2}{\kappa_{min}}\|\nabla T_t\|^2_{L^2(0,t^{\ast};L^2(\Omega))} + \|\psi_h^{0}\|^2 + \Delta t\|\sqrt{\kappa_{max}}\nabla\psi_h^{0}\|^2.
\end{aligned}
\end{eqnarray}

Collecting constants, application of Lemma \ref{lemma_gronwall}, and a rearrangement yield
\begin{eqnarray}
\begin{aligned}
\label{dingli2.9}
\|\psi_h^{N}\|^2 + \Delta t\|\sqrt{\kappa_{max}}\nabla\psi_h^{N}\|^2 + \sum^{N-1}_{n=0} \big(\|\psi_h^{n+1}-\psi_h^{n}\|^2 + \Delta t\|\sqrt{\kappa^{\prime n}_h}\nabla (\psi^{n+1}_h-\psi^{n}_h)\|^2\big)+ \frac{\Delta t}{2}\sum^{N-1}_{n=0} \|\sqrt{\kappa_h^n}\nabla \psi^{n+1}_h\|^2
\\ \leq C\exp\Big(\frac{20C^2_{\kappa}}{\kappa_{min}}\Delta t\sum^{N-1}_{n=0}\|\nabla T^{n}\|^2_{\infty}\Big)\Big\{\kappa_{min}^{-1}\Big(\|\phi_t\|^2_{L^2(0,t^{\ast},L^2(\Omega))} + \kappa_{max}^2 \vertiii{\phi}^2_{2,0} + \max_{0\leq n \leq N}\|\nabla T^{n}\|^2_{\infty}\vertiii{\nabla\phi}^2_{2,0}\Big)
\\ + \Delta t^2\Big(\kappa_{min}^{-1}\|T_{tt}\|^2_{L^2(0,t^{\ast};L^{2}(\Omega))} + \kappa_{min}^{-1}\max_{0\leq n \leq N}\|\nabla T^{n}\|^2_{\infty}\|T_t\|^2_{L^2(0,t^{\ast};L^2(\Omega))} + \kappa_{max}^2\kappa_{min}^{-1}\|\nabla T_t\|^2_{L^2(0,t^{\ast};L^2(\Omega))}\Big)
\\ + \|\psi_h^{0}\|^2 + \Delta t\|\sqrt{\kappa_{max}}\nabla\psi_h^{0}\|^2\Big\}.
\end{aligned}
\end{eqnarray}

Denote $C_{\dagger}=\frac{20C^2_{\kappa}}{\kappa_{min}}\Delta t\sum^{N-1}_{n=0}\|\nabla T^{n}\|^2_{\infty}$, take an infimum over $Y_{h}$, apply approximation property (\ref{bijin1}), and collect constants.  Then,
\begin{eqnarray}
\begin{aligned}
\label{dingli2.10}
\|\psi_h^{N}\|^2 + \Delta t\|\sqrt{\kappa_{max}}\nabla\psi_h^{N}\|^2 + \sum^{N-1}_{n=0} \big(\|\psi_h^{n+1}-\psi_h^{n}\|^2 + \Delta t\|\sqrt{\kappa^{\prime n}_h}\nabla (\psi^{n+1}_h-\psi^{n}_h)\|^2\big) + \frac{\Delta t}{2}\sum^{N-1}_{n=0} \|\sqrt{\kappa_h^n}\nabla \psi^{n+1}_h\|^2 
\\\leq C\exp(C_{\dagger})\Big\{\kappa_{min}^{-1}\big(1 + \kappa_{max}^2 + \kappa_{min}\big)h^{2k+2} + \kappa_{min}^{-1} h^{2k} + \big(1 + \kappa_{min}^{-1} + \kappa_{max}^2 \kappa_{min}^{-1} + \kappa_{max}\Big)\Delta t^2 \big\}.
\end{aligned}
\end{eqnarray}
Application of the triangle inequality yields the result, estimate (\ref{dingli2.0}).  For the latter result (\ref{dingli2.01}) pertaining to Robin boundary conditions, we have the following estimate for the boundary term: 
\begin{eqnarray}
\begin{aligned}
\label{dingli2.17}
&2\Delta t (\alpha\phi^{n+1},\psi^{n+1}_h)_{\partial\Omega}\leq \frac{\|\alpha\|_{L^{\infty}(\partial\Omega)}^2\Delta t}{2\kappa_{min}\epsilon_0}\|\phi_h^{n+1}\|^2_{-1,\partial \Omega} + \frac{\epsilon_0\Delta t}{2}\| \sqrt{\kappa^n_h}\psi_h^{n+1}\|^2_1.
\end{aligned}
\end{eqnarray}

Using the above estimate, noting that $\|\nabla S\| \leq \|S\|_1$ and selecting $\epsilon_i = C_P^2/40$ for $i=5,6,7,8,9$ and $\epsilon_{i} = C_P^2/16$ for $i=0,10,11,12$, leads to the result.
\end{proof}
The above Theorem allows for Lagrange elements of arbitrary polynomial order to be used.  However, if $P1$ is used, the optimal, first-order accuracy is achieved with $h = \mathcal{O}(\Delta t)$.
\begin{corollary}
Suppose the finite element space $Z = Y_h$ or $X_h$ is given by $P1$.  Then, under the assumptions of Theorem $\ref{theorem 2}$, the error for Algorithms 1(a) and (b) satisfy
\begin{eqnarray}
\begin{aligned}
\label{last1}
\|e^{N}\|^2 + \Delta t\|\sqrt{\kappa_{max}}\nabla e^{N}\|^2 + \sum^{N-1}_{n=0} \big(\|e^{n+1}-e^{n}\|^2 + \Delta t\|\sqrt{\kappa^{\prime n}_h}\nabla (e^{n+1}-e^{n})\|^2\big)+ \frac{\Delta t}{2}\sum^{N-1}_{n=0} \|\sqrt{\kappa_h^n}\nabla e^{n+1}\|^2
\\ \leq C(h^2 + \Delta t^2)
\end{aligned}
\end{eqnarray}
and 
\begin{eqnarray}
\begin{aligned}
\label{last2}
\|e^{N}\|^2 + \Delta t\|\sqrt{\kappa_{max}}\nabla e^{N}\|^2 + \sum^{N-1}_{n=0} \big(\|e^{n+1}-e^{n}\|^2 + \Delta t\|\sqrt{\kappa^{\prime n}_h}\nabla (e^{n+1}-e^{n})\|^2\big) + \frac{C_P^2\Delta t}{8}\sum^{N-1}_{n=0} \|\sqrt{\kappa_h^n}\nabla e^{n+1}\|^2
\\ \leq C(h^2 + \Delta t^2).
\end{aligned}
\end{eqnarray}
\end{corollary}

\section{Numerical Experiments}


In this section, we illustrate the stability and convergence of the numerical schemes (\ref{quan1}) and (\ref{quan2}) using P1 Lagrange elements to approximate temperature distributions.  The first numerical experiment tests for convergence and considers the effects of ensemble and perturbation sizes, where an analytical solution is constructed via the method of manufactured solutions.  From this test it is shown that the numerical methods (\ref{quan1}) and (\ref{quan2}) are first-order accurate in the appropriate norms.  Moreover, the ensemble and perturbation sizes have little effect on the accuracy in this setting.  The next numerical experiment is a 3D printing application in the spirit of the work by Vora and Dahotre \cite{Vora} using a temperature-dependent thermal conductivity that resembles data from \cite{kolev2015}.  A simple application to uncertainty quantification, calculation of error envelopes, is illustrated using temperature-dependent thermal conductivity.  The final validation experiment is the steady-state solution of a nonlinear heat transfer problem from \cite{singh2006numerical} which is compared to a given analytical solution. The software used for all tests is \textsc{FreeFem}$++$ \cite{Hecht}.



\begin{table}\centering
	\caption{Errors and rates for algorithm (\ref{quan1})}. 
	\begin{tabular}{ c  c  c  c  c }
		\hline			
		$m$ & $\vertiii{ <e>  }_{\infty,0}$ & Rate & $\vertiii{ \nabla <e> }_{2,0}$ & Rate \\
		\hline
        4&	1.81E-02&	-&	2.55E-01&-	\\
        8&	4.37E-03&	2.06&	1.27E-01&	1.00 \\
        16&	1.11E-03&	1.98&	6.04E-02&	1.08 \\
        32&	3.13E-04&	1.83&	3.04E-02&	0.99 \\
        64&	9.07E-05&	1.79&	1.55E-02&	0.98 \\
		\hline  
	\end{tabular}\label{Table=first}
\end{table}

\begin{table}\centering
	\caption{Errors and rates for algorithm (\ref{quan2})}.
	\begin{tabular}{ c  c  c  c  c }
		\hline			
		$m$ & $\vertiii{ <e> }_{\infty,0}$ & Rate & $\vertiii{ \nabla <e> }_{2,0}$ & Rate \\
		\hline
        4&	1.85E-02&-&		2.50E-01&-	 \\
        8&	4.17E-03&	2.16&	1.29E-01&	0.96 \\
        16&	1.19E-03&	1.82&	6.07E-02&	1.09 \\
        32&	4.47E-04&	1.42&	3.05E-02&	1.00 \\
        64&	1.94E-04&	1.20	&1.55E-02	&0.98 \\
		\hline  
	\end{tabular}\label{Table=second}
\end{table}



\subsection{Numerical convergence study}
For the first numerical experiment we will illustrate the convergence rates for the proposed algorithms (\ref{quan1}) and (\ref{quan2}).  Let the domain $\Omega$ be the unit square $[0,1]^{2}$ and final time $t^{*} = 1$; see Figure \ref{Figure=domain} for the domain and boundary conditions.  Let $c = -0.1$, with $J = 4$ and $T(x,y,t,\omega_{j}) = (1 + \epsilon_{j})T(x,y,t)$, where $\epsilon_{j} = \mathcal{O}(10^{-1})$ for $1\leq j \leq 4$.  The manufactured solution and thermal conductivity are
\begin{align*}
T(x,y,t) &= 20\cos(t)\left(\cos(x(x-1))\sin(y(y-1)) - y(y-1)\right)
\\ \kappa(T) &= \exp(c T),
\end{align*}
where both the heat source and boundary terms are adjusted appropriately. For algorithm (\ref{quan2}), the Robin boundary condition is used with $\alpha=1/2$ and appropriate $\beta$.\\

\textbf{Remark:}  The perturbations are randomly generated.  For the first test, they are 0.9578666373, 0.9721124752, 0.35623152985, and 0.4332194024.\\

The finite element mesh $\Omega_{h}$ is a Delaunay triangulation generated from $m$ points on each side of $\Omega$.  P1 Lagrange elements are used.  We calculate the error in the approximations of the average temperature with the $L^{\infty}(0,t^{\ast};L^{2}(\Omega))$ and $L^{2}(0,t^{\ast};H^{1}(\Omega))$ norms.  Rates are calculated from the errors at two $\Delta t_{1,2}$ via
\begin{align*}
\frac{\log_{2}(e(\Delta t_{1})/e(\Delta t_{2}))}{\log_{2}(\Delta t_{1}/\Delta t_{2})}.
\end{align*}
We set $\Delta t = 0.5/m$ and vary $m$ between 4, 8, 16, 32, and 64.  Results are presented in Tables \ref{Table=first} and \ref{Table=second}.  We see first-order convergence in the $L^{\infty}(0,t^{\ast};L^{2}(\Omega))$ norm and first-order convergence in the $L^{2}(0,t^{\ast};H^{1}(\Omega))$ norm for each algorithm. These results are as expected based on the convergence analysis, Theorem \ref{theorem 2}.

For the second test, we study the effect of the size of the perturbation on convergence.  We repeat the above convergence test, changing only the perturbation size.  For $1\leq j \leq 4$, let $\epsilon_{j} = \mathcal{O}(10^{-l})$ for $l = 0, 1, 2, 3$, and 4.  The errors of the average solution in $L^{\infty}(0,t^{\ast};L^{2}(\Omega))$ are presented in Tables \ref{Table=perturb1} and \ref{Table=perturb2} for methods (\ref{quan1}) and (\ref{quan2}), respectively.  We see that as the perturbation size is reduced, the results increasingly agree with one another.  Notably, the algorithm remains stable irrespective of the perturbation size, consistent with Theorem \ref{theorem1}.

Finally, we investigate the effect of the ensemble size J. We fix $m=16$, $\Delta{t}=0.5/m$, $\epsilon_j = \mathcal{O}(10^{-1})$, and then let $J$ vary from 1, 2, 4, 8, 16, 32, and 64. The associated average errors are calculated and plotted, Figure \ref{Figure=Jerror}. Once again, results are consistent with the theory.

\begin{table}\centering
	\caption{Comparison of $\vertiii{<e>}_{\infty,0}$ with algorithm (\ref{quan1}), varying perturbation size, $\epsilon$.} 
	\begin{tabular}{ c  c  c  c  c c  c  c  c	c}
		\hline
		\multicolumn{1}{c} {Mesh} & \multicolumn{5}{c} {Perturbation Size} \\			
		$m$ &  $\mathcal{O}(1)$ & $\mathcal{O}(10^{-1})$ & $\mathcal{O}(10^{-2})$ & $\mathcal{O}(10^{-3})$ & $\mathcal{O}(10^{-4})$ \\
		\hline
        4&	8.16E-02&	1.81E-02&	1.20E-02&	1.14E-02&	1.13E-02 \\
        8&	1.94E-02&	4.37E-03&	2.89E-03&	2.74E-03&	2.73E-03 \\
        16&	4.92E-03&	1.11E-03&	7.35E-04&	6.97E-04&	6.94E-04 \\
        32&	1.37E-03&	3.13E-04&	2.08E-04&	1.97E-04&	1.96E-04 \\
        64&	3.83E-04&	9.07E-05&	6.05E-05&	5.75E-05&	5.71E-05 \\
		\hline  
	\end{tabular}\label{Table=perturb1}
\end{table}

\begin{table}\centering
	\caption{Comparison of $\vertiii{<e>}_{\infty,0}$ with algorithm (\ref{quan2}), varying perturbation size, $\epsilon$.} 
	\begin{tabular}{ c  c  c  c  c c  c  c  c	c}
		\hline
		\multicolumn{1}{c} {Mesh} & \multicolumn{5}{c} {Perturbation Size} \\			
		$m$ &  $\mathcal{O}(1)$ & $\mathcal{O}(10^{-1})$ & $\mathcal{O}(10^{-2})$ & $\mathcal{O}(10^{-3})$ & $\mathcal{O}(10^{-4})$ \\
		\hline
        4&	8.27E-02&	1.85E-02&	1.23E-02&	1.16E-02&	1.16E-02 \\
        8&	1.85E-02&	4.17E-03&	2.76E-03&	2.62E-03&	2.61E-03 \\
        16&	4.93E-03&	1.19E-03&	7.92E-04&	7.52E-04&	7.48E-04 \\
        32&	1.76E-03&	4.47E-04&	2.99E-04&	2.84E-04&	2.83E-04 \\
        64&	7.63E-04&	1.94E-04&	1.30E-04&	1.24E-04&	1.23E-04 \\
		\hline  
	\end{tabular}\label{Table=perturb2}
\end{table}

\begin{table}\centering
    \caption{Comparison of steady-state solutions with exact solution.}
    \begin{tabular}{l l c c c c c c}
        \hline
        \multicolumn{2}{l} {Position} & & \multicolumn{2}{c} {$m=8$} & \multicolumn{2}{c} {$m=16$} \\
        $x$  & $y$ & & $T$ & \% Error & $T$ & \% Error & Analytical  \\\hline
        0.25 & 0.50 & &161.939 & 0.022 & 161.919 & 0.009 & 161.904 \\
        0.375 & 0.630 & &143.281 & 0.020 & 143.259 & 0.004 & 143.253 \\
        0.50 & 0.50 & &132.309 & 0.016 & 132.293 & 0.004 & 132.288 \\
        0.50 & 0.75 & &124.361 & 0.015 & 124.342 & 0.000 & 124.342 \\
        0.625 & 0.625 & &120.343 & 0.013 & 120.332 & 0.003 & 120.328 \\
        0.75 & 0.50 & &113.423 & 0.009 & 113.415 & 0.002 & 113.413 \\
        0.75 & 0.75 & &109.731 & 0.007 & 109.725 & 0.002 & 109.723 \\
        0.25 & 0.75 & &151.584 & 0.043 & 151.541 & 0.015 & 151.519 \\\hline
    \end{tabular}
    \label{Table=stdystate}
\end{table}

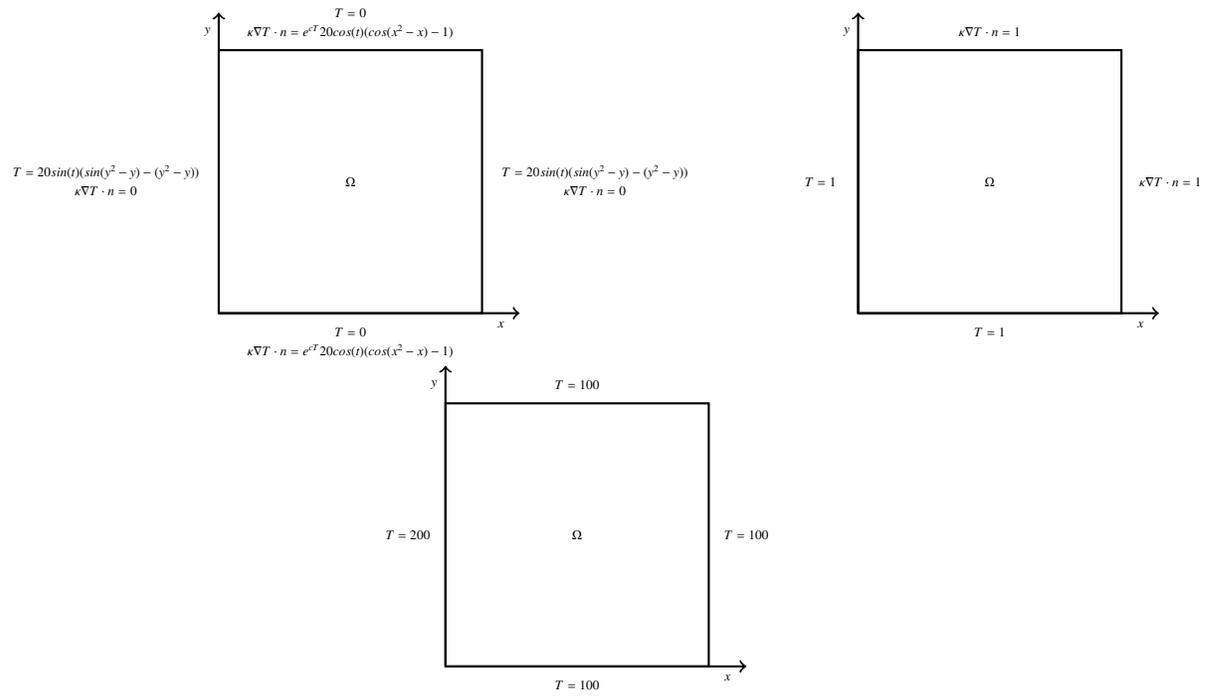
\begin{figure}
    \centering
    \begin{tikzpicture}[scale=0.5, every node/.style={scale=0.5}]
        
        \draw[->,thick] (0,0) -- (0,8);
        \draw[->,thick] (0,0) -- (8,0);
        
        \draw[black,thick] (0,0) rectangle (7,7);
        
        \node at (3.5,3.5) {$\Omega$};
        \node at (-0.3,7.5) {$y$};
        \node at (7.5,-0.3) {$x$};
        \node at (3.5,8.0) {$T = 0$};
        \node at (3.5,7.5) {$\kappa\nabla{T}\cdot{n} = e^{cT}20cos(t)(cos(x^2-x) - 1)$};
        \node at (3.5,-0.5) {$T = 0$};
        \node at (3.5,-1.0) {$\kappa\nabla{T}\cdot{n} = e^{cT}20cos(t)(cos(x^2-x) - 1)$};
        \node at (10.0,3.75) {$T = 20sin(t)(sin(y^2-y) - (y^2-y))$};
        \node at (10.0,3.25) {$\kappa\nabla{T}\cdot{n} = 0$};
        \node at (-3,3.75) {$T = 20sin(t)(sin(y^2-y) - (y^2-y))$};
        \node at (-3,3.25) {$\kappa\nabla{T}\cdot{n} = 0$};

        \draw[->,thick] (17,0) -- (17,8);
        \draw[->,thick] (17,0) -- (25,0);
        
        \draw[black,thick] (17,0) rectangle (24,7);
        
        \node at (20.5,3.5) {$\Omega$};
        \node at (16.7,7.5) {$y$};
        \node at (24.5,-0.3) {$x$};
        \node at (20.5,7.5) {$\kappa\nabla{T}\cdot{n} = 1$};
        \node at (20.5,-0.5) {$T = 1$};
        \node at (25.3,3.5) {$\kappa\nabla{T}\cdot{n} = 1$};
        \node at (16,3.5) {$T = 1$};
    \end{tikzpicture}
    
    \begin{tikzpicture}[scale=0.5, every node/.style={scale=0.5}]
        
        \draw[->,thick] (0,0) -- (0,8);
        \draw[->,thick] (0,0) -- (8,0);
        
        \draw[black,thick] (0,0) rectangle (7,7);
        
        \node at (3.5,3.5) {$\Omega$};
        \node at (-0.3,7.5) {$y$};
        \node at (7.5,-0.3) {$x$};
        \node at (3.5,7.5) {$T = 100$};
        \node at (3.5,-0.5) {$T = 100$};
        \node at (8.0,3.5) {$T = 100$};
        \node at (-1,3.5) {$T = 200$};
    
    \end{tikzpicture}
    \caption{Domain and boundary conditions for (left) convergence test manufactured solution, (right) 3D printing problem, and (bottom) steady-state solution.}
    \label{Figure=domain}
\end{figure}

\begin{figure}
    \centering
    \includegraphics[scale = 0.55]{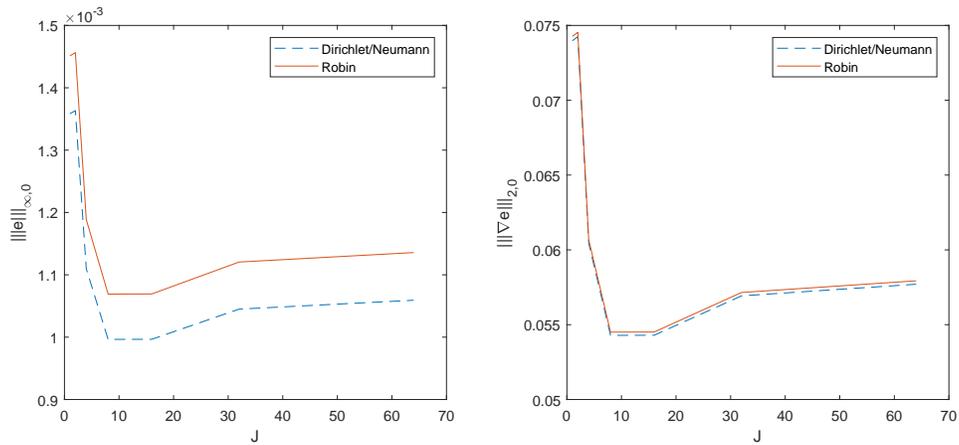}
    \caption{Errors versus ensemble size J.}
    \label{Figure=Jerror}
\end{figure}

\begin{figure}
    \centering
    \includegraphics[scale = 0.75]{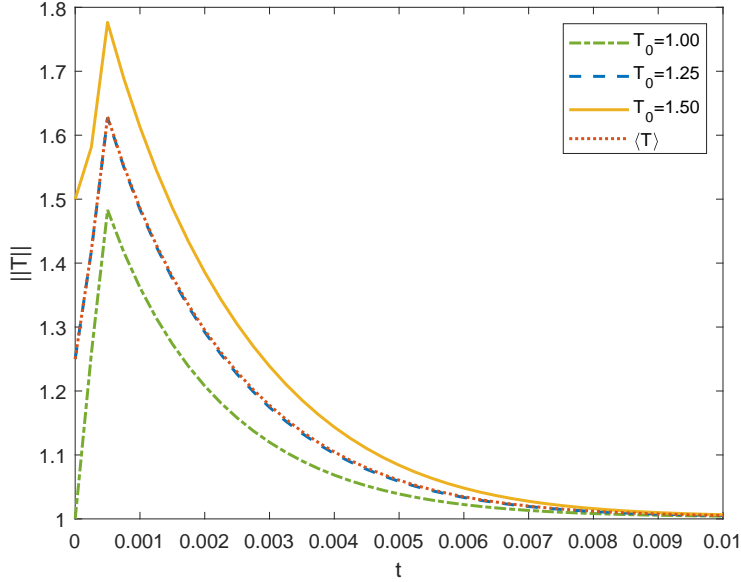}
    \caption{Variation of $||T||$ with time.}
    \label{Figure=printingex}
\end{figure}

\subsection{3D printing application}
We now consider an application problem in the spirit of \cite{Vora} to illustrate the use of ensembles.  The problem is the two-dimensional heat transfer of a solid medium subject to laser heating from above by a single pulse.  Emulating the thermal conductivity found in \cite{kolev2015} we let $\kappa(T) = 100(T-2.0)^2\mathcal{H}(2.0-T) + 50$ where $\mathcal{H}$ is the Heaviside step function.  Moreover, we set $J = 3$ whereby the initial conditions $T(x,y,0;\omega_{j}) = T_{0,j} = 1.0, 1.25$, and $1.5$ for $1\leq{j}\leq{3}$.  

The lower left corner walls are maintained at temperature $T(1,y,t;\omega_{j}) = T(x,0,t;\omega_{j}) = 1.0$ and upper right corner walls allow for heat flow out of the element via $\kappa \nabla T \cdot n = 1$; see Figure \ref{Figure=domain}.  
Moreover, the heat source, $f(x,y,t;\omega_{j})$, is given by
	\[ f(x,y,t;\omega_{j}) = 
	\begin{cases} 
	4000 \exp(-8((x-0.5)^2 + (y-0.5)^2)) & 0 \leq t \leq 0.0005, \\
	0 & 0.0005 < t,
	\end{cases}
	\]
representing a pulse laser with Gaussian beam profile.
\\ \indent The finite element mesh is a division of $(0,1)^{2}$ into $64^{2}$ squares with diagonals connected with a line within each square in the same direction.  We use algorithm (\ref{quan1}) with timestep $\Delta t = 0.00025$ and final time $t^{\ast} = 0.01$.
The values for each computed approximate temperature distributions and mean distribution in the $L^{2}(\Omega)$ norm are computed and presented in Figure \ref{Figure=printingex}.  We see that the temperature approximation generated by the unperturbed thermal conductivity and the mean are close as expected.  Moreover, the temperature approximations generated by perturbed thermal conductivities envelop the mean, evidently useful in quantifying uncertainty.

\subsection{Steady State Experiment}
The final numerical experiment is the solution of a two-dimensional steady-state nonlinear heat transfer problem with temperature dependent thermal conductivity as performed in \cite{singh2006numerical}. We use a single ensemble with $J=1$ and initial condition $T(x,y,0) = 100$. The left wall is set at $T(0,y,t) = 200$ and the remaining boundaries are $T(1,y,t)=T(x,0,t)=T(x,1,t)=100$. We use the heat source $f(x,y,t)=0$, and set $\kappa(T) = \frac{\kappa_0}{c\rho}T$ where $\kappa_0=400$ is the reference thermal conductivity, and the specific heat and density are set to be $c=400$ and $\rho = 9000$ respectively. These boundary conditions can be seen in Figure \ref{Figure=domain}. 
Q
The finite element mesh is a uniform division of the domain $(0,1)^2$ into $8^2$ and $16^2$ squares whose diagonals are connected with a line in the same direction for each square. Values of the steady-state solution at each mesh size are approximated and presented with a comparison to the analytical solution given from \cite{singh2006numerical} in Table \ref{Table=stdystate}; from this we can see Algorithm (\ref{quan1}) reproduces the steady-state solution with high accuracy.

\section{Conclusion}
We presented two algorithms for calculating an ensemble of solutions to heat conduction problems with uncertain temperature-dependent thermal conductivity.  In particular, these algorithms required the solution of a linear system, involving a shared coefficient matrix, for multiple right-hand sides at each timestep.  Unconditional stability and convergence of the algorithms were proven.  Moreover, numerical experiments were performed to illustrate the use of ensembles and the proven properties.  Important next steps include allowing for phase changes in the solid material (e.g., liquid phase) and incorporating more physics in the boundary conditions (e.g., surface-to-ambient radiation).

\section*{Acknowledgements}
J. A. Fiordilino and M. Winger are supported by NISE/Section 219.

\section*{Public Release}
Distribution Statement A: Approved for public release; distribution is unlimited. (NSWC Corona Public Release Control Number 21-012).

\end{document}